\documentclass[10pt,reqno]{amsart}
\usepackage{amssymb,amscd}
\usepackage{color}   

\newtheorem{theorem}{Theorem}[section]                    
\newtheorem{proposition}[theorem]{Proposition}            
                
\newtheorem{lemma}[theorem]{Lemma}
\newtheorem{remark}[theorem]{Remark}

\begin{document}
\title{A Remark on Orbital Free Entropy}
\author[Y.~Ueda]{Yoshimichi Ueda}
\address{
Graduate School of Mathematics, 
Kyushu University, 
Fukuoka, 810-8560, Japan
}
\email{ueda@math.kyushu-u.ac.jp}
\thanks{Supported by Grant-in-Aid for Challenging Exploratory Research 16K13762.}
\subjclass[2010]{46L54, 52C17, 28A78, 94A17.}
\keywords{Free independence; Free entropy; Free mutual information; Orbital free entropy}
\date{Feb. 20th, 2017}

\begin{abstract} A lower estimate of the orbital free entropy $\chi_\mathrm{orb}$ under unitary conjugation is proved, and it together with Voiculescu's observation shows that the conjectural  exact formula relating $\chi_\mathrm{orb}$ to the free entropy $\chi$ breaks in general in contrast to the case when given random multi-variables are all hyperfinite. 
\end{abstract} 

\maketitle

\allowdisplaybreaks{

\section{Introduction} 

Voiculescu's theory of free entropy (see \cite{Voiculescu:Survey}) has two alternative approaches; the microstate free entropy $\chi$ and the microstate-free free entropy $\chi^*$, both of which are believed to define the \emph{same} free entropy (at least under the $R^\omega$-embeddability assumption). Similarly to the microstate free entropy $\chi$, the orbital free entropy $\chi_\mathrm{orb}(\mathbf{X}_1,\dots,\mathbf{X}_n)$ of given random self-adjoint multi-variables $\mathbf{X}_i$ was also constructed based on an appropriate notion of microstates (called `orbital microstates', i.e., the `unitary orbit part' of the usual matricial microstates appearing in the definition of $\chi$) in \cite{HiaiMiyamotoUeda:IJM09},\cite{Ueda:IUMJ14}. The free entropy should be understood, in some senses, as the `size' of a given set of non-commutative random variables, while $\chi_\mathrm{orb}(\mathbf{X}_1,\dots,\mathbf{X}_n)$ \emph{precisely} measures how far the positional relation among the $W^*(\mathbf{X}_i)$ are from the freely independent positional relation in the ambient tracial $W^*$-probability space. In fact, we have known that $\chi_\mathrm{orb}(\mathbf{X}_1,\dots,\mathbf{X}_n)$ is non-positive and equals zero if and only if the $W^*(\mathbf{X}_i)$ are freely independent (modulo the $R^\omega$-embeddability assumption). This fact and the other general properties of $\chi_\mathrm{orb}$ suggest that the minus orbital free entropy $-\chi_\mathrm{orb}$ is a microstate variant of Voiculescu's free mutual information $i^*$ whose definition is indeed `microstate-free'. Hence it is natural to expect that those two quantities have the same properties. 

In \cite{Voiculescu:AdvMath99} Voiculescu (implicitly) proved that 
$$
i^*(v_1 A_1 v_1^*;\dots;v_n A_n v_n^*: B) \leq -(\Sigma(v_1)+\cdots+\Sigma(v_n))
$$ 
holds for unital $*$-subalgebras $A_1,\dots,A_n,B$ of a tracial $W^*$-probability space and a freely independent family of unitaries $v_1,\dots,v_n$ in the same $W^*$-probability space such that the family is freely independent of $A_1\vee\cdots\vee A_n \vee B$. Here, we set $\Sigma(v_i) := \int_\mathbb{T}\int_\mathbb{T}\log|\zeta_1 - \zeta_2|\,\mu_{v_i}(d\zeta_1)\mu_{v_i}(d\zeta_2)$ with the spectral distribution measure $\mu_{v_i}$ of $v_i$ with respect to $\tau$. In fact, this inequality immediately follows from \cite[Proposition 9.4]{Voiculescu:AdvMath99} (see Proposition 10.11 in the same paper). Its natural `orbital counterpart' should be
$$
\chi_\mathrm{orb}(v_1\mathbf{X}_1 v_1^*,\dots,v_n\mathbf{X}_n v_n^*) \geq \Sigma(v_1)+\cdots+\Sigma(v_n)
$$ 
with regarding $A_i = W^*(\mathbf{X}_i)$ ($1 \leq i \leq n$) and $B = \mathbb{C}$. We will prove a slightly improved inequality (Theorem \ref{T1}). The inequality is nothing but a further evidence about the unification conjecture between $i^*$ and $\chi_\mathrm{orb}$. However, more importantly, the inequality together with Voiculescu's discussion \cite[\S\S14.1]{Voiculescu:AdvMath99} answers, in the negative, the question on the expected relation between $\chi_\mathrm{orb}$ and $\chi$. Namely, the main formula in \cite{HiaiMiyamotoUeda:IJM09} (see \eqref{Eq7} below), which we call the exact formula relating $\chi_\mathrm{orb}$ to $\chi$, does not hold without any additional assumptions. 

In the final part of this short note we also give an observation about the question of whether or not there is a variant of $\chi_\mathrm{orb}$ satisfying both the `$W^*$-invariance' for each given random self-adjoint multi-variable and the exact formula relating $\chi_\mathrm{orb}$ to $\chi$ in general. Here it is fair to mention two other attempts due to Biane--Dabrowski \cite{BianeDabrowski:AdvMath13} and Dabrowski \cite{Dabrowki:arXiv:1604.06420}, but this question is not yet resolved at the moment of this writing.

\section{Preliminaries} 

Throughout this note, $(\mathcal{M},\tau)$ denotes a tracial $W^*$-probability space, that is, $\mathcal{M}$ is a finite von Neumann algebra and $\tau$ a faithful normal tracial state on $\mathcal{M}$. We denote the $N\times N$ self-adjoint matrices by $M_N(\mathbb{C})^\mathrm{sa}$ and the Haar probability measure on the $N\times N$ unitary group $\mathrm{U}(N)$ by $\gamma_{\mathrm{U}(N)}$. 

\subsection{Orbital free entropy} (\cite{HiaiMiyamotoUeda:IJM09},\cite{Ueda:IUMJ14}.) Let $\mathbf{X}_i = (X_{i1},\dots,X_{ir(i)})$, $1 \leq i \leq n$, be arbitrary random self-adjoint multi-variables in $(\mathcal{M},\tau)$. We recall an expression of $\chi_\mathrm{orb}(\mathbf{X}_1,\dots,\mathbf{X}_n)$ that we will use in this note. Let $R > 0$ be given possibly with $R=\infty$, and $m \in \mathbb{N}$ and $\delta > 0$ be arbitrarily given. For given multi-matrices $\mathbf{A}_i = (A_{ij})_{j=1}^{r(i)} \in (M_N(\mathbb{C})^\mathrm{sa})^{r(i)}$, $1 \leq i \leq n$, the set of orbital microstates $\Gamma_\mathrm{orb}(\mathbf{X}_1,\dots,\mathbf{X}_n:(\mathbf{A}_i)_{i=1}^n\,;\,N,m,\delta)$ is defined to be all $(U_i)_{i=1}^n \in \mathrm{U}(N)^n$ such that  
$$
\big|\mathrm{tr}_N(h((U_i \mathbf{A}_i U_i^*)_{i=1}^n)) - \tau(h((\mathbf{X}_i)_{i=1}^n))\big| < \delta
$$
holds whenever $h$ is a $*$-monomial in $(r(1)+\cdots+r(n))$ indeterminates of degree not greater than $m$. Similarly, $\Gamma_R(\mathbf{X}_i;N,m,\delta)$ denotes the set of all $\mathbf{A} \in ((M_N(\mathbb{C})^\mathrm{sa})_R)^{r(i)}$ such that 
$$
\big|\mathrm{tr}_N(h(\mathbf{A})) - \tau(h(\mathbf{X}_i))\big| < \delta
$$
holds whenever $h$ is a $*$-monomial in $r(i)$ indeterminates of degree not greater than $m$. It is rather trivial that if some $\mathbf{A}_i$ sits in $((M_N(\mathbb{C})^\mathrm{sa})_R)^{r(i)}\setminus\Gamma_R(\mathbf{X}_i\,;\,N,m,\delta)$, then $\Gamma_\mathrm{orb}(\mathbf{X}_1,\dots,\mathbf{X}_n:(\mathbf{A}_i)_{i=1}^n\,;\,N,m,\delta)$ must be the empty set. Hence we define
\begin{align}\label{Eq1}
&\bar{\chi}_{\mathrm{orb},R}(\mathbf{X}_1,\dots,\mathbf{X}_n\,;\,N,m,\delta) \notag\\
&\quad\quad:= 
\sup_{\mathbf{A}_i \in (M_N(\mathbb{C})^\mathrm{sa})_R^{r(i)}}\log\Big(\gamma_{\mathrm{U}(N)}^{\otimes n}\big(\Gamma_\mathrm{orb}(\mathbf{X}_1,\dots,\mathbf{X}_n:(\mathbf{A}_i)_{i=1}^n\,;\,N,m,\delta)\big)\Big) \\
&\quad\quad\,= 
\sup_{\mathbf{A}_i \in \Gamma_R(\mathbf{X}_i\,;\,N,m,\delta)}\log\Big(\gamma_{\mathrm{U}(N)}^{\otimes n}\big(\Gamma_\mathrm{orb}(\mathbf{X}_1,\dots,\mathbf{X}_n:(\mathbf{A}_i)_{i=1}^n\,;\,N,m,\delta)\big)\Big) \notag
\end{align}
(defined to be $-\infty$ if some $\Gamma_R(\mathbf{X}_i\,;\,N,m,\delta) = \emptyset$), and we define  
\begin{equation}\label{Eq2}
\chi_{\mathrm{orb},R}(\mathbf{X}_1,\dots,\mathbf{X}_n) 
:= 
\lim_{\substack{m\to\infty \\ \delta\searrow0}} \varlimsup_{N\to\infty} 
\frac{1}{N^2}\bar{\chi}_{\mathrm{orb},R}(\mathbf{X}_1,\dots,\mathbf{X}_n\,;\,N,m,\delta). 
\end{equation} 
It is known, see \cite[Corollary 2.7]{Ueda:IUMJ14}, that $\chi_\mathrm{orb}(\mathbf{X}_1,\dots,\mathbf{X}_n) := \sup_{R > 0}\chi_{\mathrm{orb},R}(\mathbf{X}_1,\dots,\mathbf{X}_n) = \chi_{\mathrm{orb},R}(\mathbf{X}_1,\dots,\mathbf{X}_n)$ holds whenever $R \geq \max\{\Vert X_{ij}\Vert_\infty\,|\,1 \leq i \leq n, 1 \leq j \leq r(i)\}$. 

Let $\mathbf{v} = (v_1,\dots,v_s)$ be an $s$-tuple of unitaries in $(\mathcal{M},\tau)$. For given multi-matrices $\mathbf{A}_i = (A_{ij})_{j=1}^{r(i)} \in (M_N(\mathbb{C})^\mathrm{sa})^{r(i)}$, $1 \leq i \leq n$, $\Gamma_\mathrm{orb}(\mathbf{X}_1,\dots,\mathbf{X}_n:(\mathbf{A}_i)_{i=1}^n:\mathbf{v}\,;\,N,m,\delta)$ in presence of $\mathbf{v}$ is defined to be all $(U_i)_{i=1}^n \in \mathrm{U}(N)^n$ such that there exists $\mathbf{V} = (V_1,\dots,V_s) \in U(N)^s$ so that 
$$
\big|\mathrm{tr}_N(h((U_i \mathbf{A}_i U_i^*)_{i=1}^n,\mathbf{V})) - \tau(h((\mathbf{X}_i)_{i=1}^n,\mathbf{v}))\big| < \delta
$$
holds whenever $h$ is a $*$-monomial in $(r(1)+\cdots+r(n)+s)$ indeterminates of degree not greater than $m$. Then $\chi_{\mathrm{orb},R}(\mathbf{X}_1,\dots,\mathbf{X}_n:\mathbf{v})$ can be obtained in the same way as above with $\Gamma_\mathrm{orb}(\mathbf{X}_1,\dots,\mathbf{X}_n:(\mathbf{A}_i)_{i=1}^n:\mathbf{v}\,;\,N,m,\delta)$ in place of $\Gamma_\mathrm{orb}(\mathbf{X}_1,\dots,\mathbf{X}_n:(\mathbf{A}_i)_{i=1}^n\,;\,N,m,\delta)$. Remark that $\chi_\mathrm{orb}(\mathbf{X}_1,\dots,\mathbf{X}_n:\mathbf{v}) := \sup_{R>0}\chi_{\mathrm{orb},R}(\mathbf{X}_1,\dots,\mathbf{X}_n:\mathbf{v}) = \chi_{\mathrm{orb},R}(\mathbf{X}_1,\dots,\mathbf{X}_n:\mathbf{v})$ also holds if $R \geq \max\{\Vert X_{ij}\Vert_\infty\,|\,1 \leq i \leq n, 1 \leq j \leq r(i)\}$. Moreover, $\chi_\mathrm{orb}(\mathbf{X}_1,\dots,\mathbf{X}_n:\mathbf{v}) \leq \chi_\mathrm{orb}(\mathbf{X}_1,\dots,\mathbf{X}_n)$ trivially holds. 

\subsection{Microstate free entropy for unitaries} (See \cite[\S6.5]{HiaiPetz:Book}.) Let $\mathbf{v} = (v_1,\dots,v_n)$ be an $n$-tuple of unitaries in $\mathcal{M}$. We recall the microstate free entropy $\chi_u(\mathbf{v})$. Let $m \in \mathbb{N}$ and $\delta > 0$ be arbitrarily given. For every $N \in \mathbb{N}$ we define $\Gamma_u(\mathbf{v};N,m,\delta)$ to be the set of all $\mathbf{V} = (V_1,\dots,V_n) \in \mathrm{U}(N)^n$ such that $\big|\mathrm{tr}_N(h(\mathbf{V})) - \tau(h(\mathbf{v}))\big| < \delta$ holds whenever $h$ is a $*$-monomial in $n$ indeterminates of degree not greater than $m$. Then 
\begin{equation}\label{Eq3} 
\chi_u(\mathbf{v}) := \lim_{m\to\infty \atop \delta\searrow0} \varlimsup_{N\to\infty} \frac{1}{N^2}\gamma_{\mathrm{U}(N)}^{\otimes n}\big(\Gamma_u(\mathbf{v};N,m,\delta)\big). 
\end{equation}
Note that $\chi_u(\mathbf{v}) = \sum_{i=1}^n\chi_u(v_i)$ holds when $v_1,\dots,v_n$ are freely independent and that $\chi_u(\mathbf{v}) = 0$ if $\mathbf{v}$ is a freely independent family of Haar unitaries. Moreover, when $n=1$, $\chi_u(v_1) = \Sigma(v_1)$ holds.        

\subsection{Voiculescu's measure concentration result} (\cite{Voiculescu:IMRN98}) Let $(\mathfrak{A},\phi)$ be a non-commutative probability space, and $(\Omega_i)_{i\in I}$ be a
family of subsets of $\mathfrak{A}$. Denote by $(\mathfrak{A}^{\star I},\phi^{\star I})$ the reduced free product of copies of $(\mathfrak{A},\phi)$ indexed by $I$, and by $\lambda_i$ the canonical map of $\mathfrak{A}$ onto the $i$-th copy of $\mathfrak{A}$ in $\mathfrak{A}^{\star I}$. For each $\varepsilon > 0$ and $m \in \mathbb{N}$ we say that $(\Omega_i)_{i\in I}$ are $(m,\varepsilon)$-free (in $(\mathfrak{A},\phi)$) if 
\begin{equation*} 
\left|\phi(a_1\cdots a_k)
- \phi^{\star I}(\lambda_{i_1}(a_1)\cdots \lambda_{i_k}(a_k))\right| < \varepsilon 
\end{equation*} 
for all $a_j \in \Omega_{i_j}$, $i_j \in I$ with $1 \leq j \leq k$ and
$1 \leq k \leq m$. 

\begin{lemma}\label{L1}{\rm(Voiculescu \cite[Corollary 2.13]{Voiculescu:IMRN98})} Let $R>0$, $\varepsilon > 0$, $\theta > 0$ and $m \in \mathbb{N}$ be given. Then there exists $N_0 \in \mathbb{N}$ such that 
\begin{align*} 
\gamma_{\mathrm{U}(N)}^{\otimes p}\big(\big\{ (U_1,\dots,U_p) \in \mathrm{U}(N)^p:
&\ \{T_1^{(0)},\dots,T_{q_0}^{(0)}\},
\{U_1 T_1^{(1)}U_1^*,\dots,U_1 T_{q_1}^{(1)}U_1^*\}, \\ 
&\dots,\{U_p T_1^{(p)}U_p^*,\dots,U_p T_{q_p}^{(p)}U_p^*\}
\ \text{are $(m,\varepsilon)$-free} \big\}\big) > 1-\theta
\end{align*} 
whenever $N \geq N_0$ and $T_j^{(i)} \in M_N(\mathbb{C})$ with
$\Vert T_j^{(i)}\Vert_\infty \leq R$, $1 \leq p \leq m$, $1 \leq q_i \leq m$, $1 \leq i \leq p$, $1 \leq j \leq q_i$. 
$0 \leq i \leq p$. 
\end{lemma}

\section{Lower estimate of $\chi_\mathrm{orb}$ under unitary conjugation} 

This section is devoted to proving the following: 

\begin{theorem}\label{T1} Let $\mathbf{X}_1,\dots,\mathbf{X}_{n+1}$ be random self-adjoint multi-variables in $(\mathcal{M},\tau)$ and $\mathbf{v} = (v_1,\dots,v_n)$ be an $n$-tuple of unitaries in $\mathcal{M}$. Assume that $\mathbf{X} := \mathbf{X}_1\sqcup\cdots\sqcup\mathbf{X}_{n+1}$ has f.d.a.~in the sense of Voiculescu \cite[Definition 3.1]{Voiculescu:IMRN98} (or equivalently, $W^*(\mathbf{X})$ is $R^\omega$-embeddable) and that $\mathbf{X}$ and $\mathbf{v}$ are freely independent. Then 
\begin{equation}\label{Eq4} 
\begin{aligned} 
\chi_\mathrm{orb}(v_1\mathbf{X}_1 v_1^*,\dots,v_n\mathbf{X}_n v_n^*,\mathbf{X}_{n+1}) 
&\geq 
\chi_\mathrm{orb}(v_1\mathbf{X}_1 v_1^*,\dots,v_n\mathbf{X}_n v_n^*,\mathbf{X}_{n+1}:\mathbf{v}) \\
&\geq 
\chi_\mathrm{orb}(v_1\mathbf{X}v_1^*,\dots,v_n\mathbf{X}v_n^*,\mathbf{X}:\mathbf{v}) \\
&\geq 
\chi_u(\mathbf{v}).
\end{aligned}
\end{equation}
\end{theorem}
\begin{proof} 
The first inequality in \eqref{Eq4} is trivial, and the second follows from (the conditional variant of) \cite[Theorem 2.6(6)]{Ueda:IUMJ14}. Hence it suffices only to prove the third inequality in \eqref{Eq4}. We may and do also assume that $\chi_u(\mathbf{v}) > -\infty$; otherwise the desired inequality trivially holds. 

Write $\mathbf{X} = (X_1,\dots,X_r)$ for simplicity. Set $R := \max\{\Vert X_j\Vert_\infty\,|\,1\leq j \leq r\}$, and let $m \in \mathbb{N}$ and $\delta>0$ be arbitrarily given. We can choose $\delta'>0$ in such a way that $\delta' \leq \delta$ and that, for every $N \in \mathbb{N}$, if $\mathbf{A} \in \Gamma_R(\mathbf{X};N,m,\delta')$ and $\mathbf{V}=(V_1,\dots,V_n) \in \Gamma_u(\mathbf{v};N,2m,\delta')$ are $(3m,\delta')$-free, then 
$$
\big|\mathrm{tr}_N(h(V_1\mathbf{A}V_1\sqcup\cdots\sqcup V_n\mathbf{A}V_n^*\sqcup\mathbf{A}\sqcup\mathbf{V})) - \tau(h(v_1\mathbf{X}v_1^*\sqcup\cdots\sqcup v_n\mathbf{X}v_n^*\sqcup\mathbf{v}))\big| < \delta
$$
whenever $h$ is a $*$-monomial of $(n+1)r + n$ indeterminates of degree not greater than $m$. For such a $\delta'>0$ the assumptions here ensure that there exists $N_0 \in \mathbb{N}$ so that $\Gamma_R(\mathbf{X}\,;\,N,m,\delta') \neq \emptyset$ and the probability measure 
$$
\nu_N := \frac{1}{\gamma_{\mathrm{U}(N)}^{\otimes n}(\Gamma_u(\mathbf{v};N,2m,\delta'))}\gamma_{\mathrm{U}(N)}^{\otimes n}\!\upharpoonright_{\Gamma_u(\mathbf{v};N,2m,\delta')}
$$
is well-defined whenever $N \geq N_0$. Let $\Xi(N) \in \Gamma_R(\mathbf{X};N,m,\delta')$ be arbitrarily chosen for each $N \geq N_0$. Note that $\Xi(N)$ also falls in $\Gamma_R(\mathbf{X};N,m,\delta) = \Gamma(v_i\mathbf{X}v_i^*;N,m,\delta)$ since $\delta' < \delta$. Then we define 
$$
\Theta(N,3m,\delta') := \{ (V_1,\dots,V_n,U) \in \mathrm{U}(N)^{n+1}\,|\,\text{$\{V_1,\dots,V_n\}$ and $U\Xi(N)U^*$ are $(3m,\delta')$-free}\}. 
$$
By what we have remarked at the beginning of this paragraph, we see that 
\begin{equation}\label{Eq5}
\begin{aligned}
&(V_1,\dots,V_n,U) \in \Theta(N,3m,\delta')\cap(\Gamma_u(\mathbf{v};N,3m,\delta')\times\mathrm{U}(N)) \\
&\Longrightarrow (V_1 U,\dots, V_n U, U) \in \Gamma_\mathrm{orb}(v_1\mathbf{X}v_1^*,\dots,v_n\mathbf{X}v_n^*,\mathbf{X}:(\Xi(N),\dots,\Xi(N)):\mathbf{v}\,;\,N,m,\delta).
\end{aligned}
\end{equation}
By Lemma \ref{L1} there exists $N_1 \geq N_0$ so that 
$$
\gamma_{\mathrm{U}(N)}\big(\{U \in \mathrm{U}(N)\,|\,(V_1,\dots,V_n,U) \in \Theta(N,3m,\delta')\}\big) > \frac{1}{2}
$$
for every $N \geq N_1$ and every $(V_1,\dots,V_n) \in \mathrm{U}(N)^n$. Consequently, we have 
\begin{equation*}
(\nu_N\otimes\gamma_{\mathrm{U}(N)})\big(\Theta(N,3m,\delta')\big) > \frac{1}{2}
\end{equation*}  
whenever $N \geq N_1$. Therefore, for every $N \geq N_1$ we have 
\begin{equation}\label{Eq6}
\begin{aligned}
&\frac{1}{2}\gamma_{\mathrm{U}(N)}^{\otimes n}\big(\Gamma_u(\mathbf{v};N,2m,\delta')\big) \\
&< 
\gamma_{\mathrm{U}(N)}^{\otimes n}\big(\Gamma_u(\mathbf{v};N,2m,\delta')\big)\times  
(\nu_N\otimes\gamma_{\mathrm{U}(N)})\big(\Theta(N,3m,\delta')\big) \\
&= 
\gamma_{\mathrm{U}(N)}^{\otimes (n+1)}\big(\Theta(N,3m,\delta')\cap(\Gamma_u(\mathbf{v};N,3m,\delta')\times\mathrm{U}(N))\big) \\
&\leq  
\gamma_{\mathrm{U}(N)}^{\otimes (n+1)}\big(\big\{(V_1,\dots,V_n,U) \in \mathrm{U}(N)^{n+1}\,|\,\\
&\qquad(V_1 U,\dots,V_n U,U) \in \Gamma_\mathrm{orb}(v_1\mathbf{X}v_1^*,\dots,v_n\mathbf{X}v_n^*,\mathbf{X}:(\Xi(N),\dots,\Xi(N)):\mathbf{v}\,;\,N,m,\delta) \big\}\big) \\
&= 
\int_{\mathrm{U}(N)}\gamma_{\mathrm{U}(N)}(dU)\,\gamma_{\mathrm{U}(N)}^{\otimes n}\big(\big\{(V_1,\dots,V_n) \in \mathrm{U}(N)^n\,|\,\\
&\qquad(V_1 U,\dots,V_n U,U) \in \Gamma_\mathrm{orb}(v_1\mathbf{X}v_1^*,\dots,v_n\mathbf{X}v_n^*,\mathbf{X}:(\Xi(N),\dots,\Xi(N)):\mathbf{v}\,;\,N,m,\delta) \big\}\big) \\
&= 
\int_{\mathrm{U}(N)}\gamma_{\mathrm{U}(N)}(dU)\,\gamma_{\mathrm{U}(N)}^{\otimes n}\big(\big\{(U_1,\dots,U_n) \in \mathrm{U}(N)^n\,|\,\\
&\qquad(U_1,\dots,U_n,U) \in \Gamma_\mathrm{orb}(v_1\mathbf{X}v_1^*,\dots,v_n\mathbf{X}v_n^*,\mathbf{X}:(\Xi(N),\dots,\Xi(N)):\mathbf{v}\,;\,N,m,\delta) \big\}\big) \\
&= 
\gamma_{\mathrm{U}(N)}^{\otimes(n+1)}\big(\Gamma_\mathrm{orb}(v_1\mathbf{X}v_1^*,\dots,v_n\mathbf{X}v_n^*,\mathbf{X}:(\Xi(N),\dots,\Xi(N)):\mathbf{v}\,;\,N,m,\delta)\big),
\end{aligned}
\end{equation} 
where the fourth line is obtained by \eqref{Eq5} and the sixth due to the right-invariance of the Haar probability measure $\gamma_{\mathrm{U}(N)}$. Hence  
\begin{align*} 
\chi_u(\mathbf{v}) 
&\leq 
\varlimsup_{N\to\infty}\frac{1}{N^2}\log\Big(\frac{1}{2}\gamma_{\mathrm{U}(N)}^{\otimes n}\big(\Gamma_u(\mathbf{v};N,2m,\delta'\big)\Big) \\
&\leq 
\varlimsup_{N\to\infty}\frac{1}{N^2}\log\gamma_{\mathrm{U}(N)}^{\otimes(n+1)}\big(\Gamma_\mathrm{orb}(v_1\mathbf{X}v_1^*,\dots,v_n\mathbf{X}v_n^*,\mathbf{X}:(\Xi(N),\dots,\Xi(N)):\mathbf{v}\,;\,N,m,\delta)\big) \\
&\leq 
\varlimsup_{N\to\infty}\frac{1}{N^2}\bar{\chi}_{\mathrm{orb},R}(v_1\mathbf{X}v_1^*,\dots,v_n\mathbf{X}v_n^*,\mathbf{X}:\mathbf{v};N,m,\delta),
\end{align*}
implying the desired inequality since $m, \delta$ are arbitrary.   
\end{proof} 

\begin{remark}\label{R1} {\rm Inequality \eqref{Eq4} is not optimal as follows. Assume that $(\mathcal{M},\tau) = (L(\mathbb{F}_r),\tau_{\mathbb{F}_r})\star(L(\mathbb{Z}_m),\tau_{\mathbb{Z}_m})$ and that $\mathbf{X}$ is the canonical free semicircular generators of $L(\mathbb{F}_r)$ and $v$ is a canonical generator of $L(\mathbb{Z}_m)$. Since $\tau(v) = 0$, one easily confirms that $v\mathbf{X}v^*$ and $\mathbf{X}$ are freely independent so that $\chi_\mathrm{orb}(v\mathbf{X}v^*,\mathbf{X}) = 0$. On the other hand, we know that $\chi_u(v) = -\infty$, since the spectral measure of $v$ has an atom.  
}
\end{remark} 

\begin{remark} \label{R2} {\rm The proof of Theorem \ref{T1} (actually, the idea of obtaining the second equality in \eqref{Eq6}) gives an alternative representation of $\bar{\chi}_{\mathrm{orb},R}(\mathbf{X}_1,\dots,\mathbf{X}_n\,;\,N,m,\delta)$: 
\begin{align*}
&\bar{\chi}_{\mathrm{orb},R}(\mathbf{X}_1,\dots,\mathbf{X}_n\,;\,N,m,\delta)\\
&\quad\quad= 
\sup_{\mathbf{A}_i \in (M_N(\mathbb{C})^\mathrm{sa})_R^{r(i)}}\log\Big(\gamma_{\mathrm{U}(N)}^{\otimes n-1}\big(\big\{ (U_i)_{i=1}^{n-1} \in \mathrm{U}(N)^{n-1}\,\big| \\
&\qquad\qquad\qquad\qquad (U_1,\dots,U_{n-1},I_N) \in \Gamma_\mathrm{orb}(\mathbf{X}_1,\dots,\mathbf{X}_n:(\mathbf{A}_i)_{i=1}^n\,;\,N,m,\delta)\big\}\big)\Big) \\
&\quad\quad\,= 
\sup_{\mathbf{A}_i \in \Gamma_R(\mathbf{X}_i\,;\,N,m,\delta)}\log\Big(\gamma_{\mathrm{U}(N)}^{\otimes n-1}\big(\big\{ (U_i)_{i=1}^{n-1} \in \mathrm{U}(N)^{n-1}\,\big| \\
&\qquad\qquad\qquad\qquad (U_1,\dots,U_{n-1},I_N) \in \Gamma_\mathrm{orb}(\mathbf{X}_1,\dots,\mathbf{X}_n:(\mathbf{A}_i)_{i=1}^n\,;\,N,m,\delta)\big\}\big)\Big),
\end{align*}
when $n \geq 2$.  
This corresponds to \cite[Remarks 10.2(c)]{Voiculescu:AdvMath99}. 
}
\end{remark} 

\section{Discussions}  

\subsection{Negative observation} In \cite[Theorem 2.6]{HiaiMiyamotoUeda:IJM09} the following formula was shown when all $\mathbf{X}_i$ are singletons: 
\begin{equation}\label{Eq7}
\chi(\mathbf{X}_1\sqcup\dots\sqcup\mathbf{X}_n) = \chi_\mathrm{orb}(\mathbf{X}_1,\dots,\mathbf{X}_n) + \sum_{i=1}^n \chi(\mathbf{X}_i).
\end{equation}
Note that the same formula trivially holds true (as $-\infty = -\infty$) even when one replaces each singleton $\mathbf{X}_i$ with a hyperfinite non-singleton $\mathbf{X}_i$, that is, $W^*(\mathbf{X}_i)$ is hyperfinite and $\mathbf{X}_i$ consists of at least two elements. Beyond the hyperfiniteness situation, inequality ($\leq$) in \eqref{Eq7} still holds (see \cite[Proposition 2.8]{Ueda:IUMJ14}), but equality unfortunately does not in general as follows. The following argument is attributed to Voiculescu \cite[\S\S14.1]{Voiculescu:AdvMath99}. Let $\mathbf{X}=(X_1,X_2)$ be a semicircular system in $\mathcal{M}$ and $v \in \mathcal{M}$ be a unitary such that $\tau(v) \neq 0$, $\chi_u(v) > -\infty$, and that $\mathbf{X}$ and $v$ are $*$-freely independent. Set $Y_i := vX_i v^*$, $i=1,2$, and $\mathbf{Y} := (Y_1,Y_2)$. By \cite[Proposition 2.5]{Voiculescu:AdvMath99} $W^*(X_1,X_2,Y_1,Y_2) = W^*(X_1,X_2,Y_1) = W^*(X_1,X_2,v)$, and hence by \cite[Proposition 3.8]{Voiculescu:InventMath94} 
$$
\chi(X_1,X_2,Y_1,Y_2) = \chi(X_1,X_2,Y_1,I) \leq \chi(X_1,X_2,Y_1) + \chi(I) = -\infty,
$$
where $I$ denotes the unit of $\mathcal{M}$. On the other hand, by Theorem \ref{T1} $\chi_\mathrm{orb}(\mathbf{X},\mathbf{Y}) \geq \chi_u(v) > -\infty$, implying that 
$$\chi(\mathbf{X}\sqcup\mathbf{Y}) = -\infty < \chi_\mathrm{orb}(\mathbf{X},\mathbf{Y}) + \chi(\mathbf{X}) + \chi(\mathbf{Y}).
$$
In particular, the quantity ``$C^\omega$" (or probably ``$C$" too) in \cite[Remark 2.9]{Ueda:IUMJ14} does not coincide with $\chi_\mathrm{orb}$ in general. An interesting question is whether or not $\chi(\mathbf{X}_1\sqcup\cdots\sqcup\mathbf{X}_n) > -\infty$ is enough to make the exact formula relating $\chi_\mathrm{orb}$ to $\chi$ hold. Note that $\chi_\mathrm{orb} = \widetilde{\chi}_\mathrm{orb}$ (Biane--Dabrowski's variant \cite{BianeDabrowski:AdvMath13}) holds under the assumption. Moreover, the orbital free entropy dimension $\delta_{0,\mathrm{orb}}(\mathbf{X},\mathbf{Y})$ must be zero in this case thanks to \cite[Proposition 4.3(5)]{Ueda:IUMJ14}, since $\chi_\mathrm{orb}(\mathbf{X},\mathbf{Y}) > -\infty$. Also $\delta_0(\mathbf{X}) = \delta_0(\mathbf{Y}) = 2$ is trivial. Note that $\chi_u(v) > -\infty$ forces that the probability distribution of $v$ has no atom. Thus, it is likely (if one believes that $\delta_0$ gives a $W^*$-invariant) that 
$$
\delta_0(\mathbf{X}\sqcup\mathbf{Y}) \overset{?}{=} 3 < 4 = \delta_{0,\mathrm{orb}}(\mathbf{X},\mathbf{Y}) + \delta_0(\mathbf{X}) + \delta_0(\mathbf{Y})
$$ 
is expected. This means that if $\delta_0(\mathbf{X}_1\sqcup\cdots\sqcup\mathbf{X}_n) = \delta_{0,\mathrm{orb}}(\mathbf{X}_1,\dots,\mathbf{X}_n) + \sum_{i=1}^n \delta_0(\mathbf{X}_i)$ held in general, then the $W^*$-invariance problem of $\delta_0$ would be resolved negatively. Hence it seems still interesting only to ask whether $\delta_0(\mathbf{X}\sqcup\mathbf{Y}) \lneqq 4$ or not.

\subsection{Other possible variants of $\chi_\mathrm{orb}$} The above discussion tells us that if a variant of $\chi_\mathrm{orb}$ satisfies Theorem \ref{T1}, then the variant does not satisfy the exact formula relating $\chi_\mathrm{orb}$ to $\chi$ in general. Following our previous work \cite{HiaiMiyamotoUeda:IJM09} with Hiai and Miyamoto one may consider the following variant of $\chi_\mathrm{orb}$: For each $1 \leq i \leq n$, we select an (operator norm-)bounded sequence $\{\Xi_i(N)\}_{N\in\mathbb{N}}$ with $\Xi_i(N) \in (M_N(\mathbb{C})^\mathrm{sa})^{r(i)}$ such that the joint distribution of $\Xi_i(N)$ under $\mathrm{tr}_N$ converges to that of $\mathbf{X}_i$ under $\tau$ as $N\to\infty$. Then we replace $\bar{\chi}_{\mathrm{orb},R}(\mathbf{X}_1,\dots,\mathbf{X}_n\,;\,N,m,\delta)$ in the definition of $\chi_\mathrm{orb}$ with 
\begin{equation*}
\begin{aligned}
&\chi_\mathrm{orb}(\mathbf{X}_1,\dots,\mathbf{X}_n:(\Xi_i(N))_{i=1}^n\,;\,N,m,\delta) \\
&\quad\quad:= 
\log\Big(\gamma_{\mathrm{U}(N)}^{\otimes n}\big(\Gamma_\mathrm{orb}(\mathbf{X}_1,\dots,\mathbf{X}_n:(\Xi_i(N))_{i=1}^n\,;\,N,m,\delta)\big)\Big),
\end{aligned}
\end{equation*}
and define 
\begin{equation*}
\begin{aligned}
&\chi_\mathrm{orb}(\mathbf{X}_1,\dots,\mathbf{X}_n:(\Xi_i(N))_{i=1}^n) \\
&\quad\quad:= 
\lim_{m\to\infty \atop \delta\searrow0} \varlimsup_{N\to\infty}\frac{1}{N^2}\chi_\mathrm{orb}(\mathbf{X}_1,\dots,\mathbf{X}_n:(\Xi_i(N))_{i=1}^n\,;\,N,m,\delta).
\end{aligned}
\end{equation*}
The conditional variant $\chi_\mathrm{orb}(\mathbf{X}_1,\dots,\mathbf{X}_n:(\Xi_i(N))_{i=1}^n:\mathbf{v})$ is defined exactly in the same fashion as $\chi_\mathrm{orb}(\mathbf{X}_1,\dots,\mathbf{X}_n:\mathbf{v})$. Then we may consider their supremum all over the possible choices of $(\Xi_i(N))_{i=1}^n$ under some suitable constraint as a variant of $\chi_\mathrm{orb}$.  

Even if the constraint of selecting sequences of multi-matrices is chosen to be the way of approximating to the freely independent copies of given random self-adjoint multi-variables, then the resulting variant of $\chi_\mathrm{orb}$ still satisfies Theorem \ref{T1}, and in turn does not satisfy the exact formula relating $\chi_\mathrm{orb}$ to $\chi$ in general. More precisely we can prove the following:

\begin{proposition}\label{P1} Let $\mathbf{X} = (X_j)_{j=1}^r$ be a random self-adjoint multi-variables in $(\mathcal{M},\tau)$ and $\mathbf{v} = (v_1,\dots,v_n)$ be an $n$-tuple of unitaries in $\mathcal{M}$. Assume that $\mathbf{X}$ has f.d.a.~(see Theorem \ref{T1}) and that $\mathbf{X}$ and $\mathbf{v}$ are freely independent. Then there exists a bounded sequence $\{(\Xi_i(N))_{i=1}^{n+1}\}_{N\in\mathbb{N}}$ with $\Xi_i(N) \in (M_N(\mathbb{C})^\mathrm{sa})^r$ such that the joint distribution of $\Xi_1(N)\sqcup\cdots\sqcup\Xi_{n+1}(N)$ under $\mathrm{tr}_N$ converges to the freely independent $n+1$ copies $\mathbf{X}_1^f\sqcup\cdots\sqcup\mathbf{X}_{n+1}^f$ of $\mathbf{X}$ {\rm(}n.b., the joint distribution of $\mathbf{X}$ is identical to that of every $v_i\mathbf{X}v_i^*${\rm)} under $\tau$ as $N\to\infty$, and moreover that
\begin{equation*}
\begin{aligned} 
&\chi_\mathrm{orb}(v_1\mathbf{X}v_1^*,\dots,v_n\mathbf{X}v_n^*,\mathbf{X}:(\Xi_i(N))_{i=1}^{n+1}) \\
&\qquad \geq 
\chi_\mathrm{orb}(v_1\mathbf{X}v_1^*,\dots,v_n\mathbf{X}v_n^*,\mathbf{X}:(\Xi_i(N))_{i=1}^{n+1}):\mathbf{v}) \geq 
\chi_u(\mathbf{v}). 
\end{aligned}
\end{equation*}
\end{proposition}
\begin{proof} 
Let $R > 0$ be sufficiently large. Since $\mathbf{X}$ has f.d.a., Lemma \ref{L1} shows that for each $m \in \mathbb{N}$ and $\delta > 0$ one has $\{((U_i)_{i=1}^{n+1},\mathbf{A}) \in \mathrm{U}(N)^{n+1}\times ((M_N(\mathbb{C})^\mathrm{sa})_R)^r \mid (U_i \mathbf{A}U_i^*)_{i=1}^{n+1} \in \Gamma_R(\mathbf{X}_1^f\sqcup\cdots\sqcup\mathbf{X}_{n+1}^f\,;\,N,m,\delta)\} \neq \emptyset$ for all sufficiently large $N \in \mathbb{N}$. By using this fact, it is easy to choose a bounded sequence $\Xi(N) \in ((M_N(\mathbb{C})^\mathrm{sa})_R)^r$ and a sequence $(W_i(N))_{i=1}^{n+1} \in \mathrm{U}(N)^{n+1}$ in such a way that both the joint distributions of $\Xi(N)$ and of $W_1(N)\Xi(N)W_1(N)^*\sqcup\cdots\sqcup W_{n+1}(N)\Xi(N)W_{n+1}(N)^*$ under $\mathrm{tr}_N$ converge to those of $\mathbf{X}$ and of $\mathbf{X}_1^f\sqcup\cdots\sqcup\mathbf{X}_{n+1}^f$, respectively,  under $\tau$ as $N\to\infty$. Set $\Xi_i(N) := W_i(N)\Xi(N)W_i(N)^*$, $1 \leq i \leq n+1$, and we will prove that $(\Xi_i(N))_{i=1}^{n+1}$ is a desired sequence.  

For given $m \in \mathbb{N}$ and $\delta >0$, we choose $0 < \delta' < \delta$ as in the proof of Theorem \ref{T1}. Let $\nu_N$ and $\Theta(N,3m,\delta')$ be also chosen exactly in the same way as in the proof of Theorem \ref{T1}. We can choose $N_0 \in \mathbb{N}$ in such a way that $\Xi(N) \in \Gamma_R(\mathbf{X}\,;\,N,m,\delta')$ and $\nu_N$ is well-defined as long as $N \geq N_0$. By the same reasoning as in the proof of Theorem \ref{T1} we have 
\begin{equation}\label{Eq8}
\begin{aligned}
&(V_1,\dots,V_n,U) \in \Theta(N,3m,\delta')\cap(\Gamma_u(\mathbf{v};N,3m,\delta')\times\mathrm{U}(N)) \\
&\Longrightarrow (V_1 U,\dots, V_n U, U) \in \Gamma_\mathrm{orb}(v_1\mathbf{X}v_1^*,\dots,v_n\mathbf{X}v_n^*,\mathbf{X}:(\Xi(N),\dots,\Xi(N)):\mathbf{v}\,;\,N,m,\delta) \\
& \Longleftrightarrow (V_1 U W_1(N)^*,\dots, V_n U W_n(N)^*, U W_{n+1}(N)^*) \\
&\qquad\qquad\qquad\qquad\qquad\in \Gamma_\mathrm{orb}(v_1\mathbf{X}v_1^*,\dots,v_n\mathbf{X}v_n^*,\mathbf{X}:(\Xi_i(N))_{i=1}^{n+1}:\mathbf{v}\,;\,N,m,\delta).
\end{aligned}
\end{equation}
As in the proof of Theorem \ref{T1} again, Lemma \ref{L1} shows that there exists $N_1 \geq N_0$ so that 
\begin{equation*}
(\nu_N\otimes\gamma_{\mathrm{U}(N)})\big(\Theta(N,3m,\delta')\big) > \frac{1}{2}
\end{equation*}  
whenever $N \geq N_1$. Therefore, for every $N \geq N_1$ we have 
\begin{align*}
&\frac{1}{2}\gamma_{\mathrm{U}(N)}^{\otimes n}\big(\Gamma_u(\mathbf{v};N,2m,\delta')\big) \\
&< 
\gamma_{\mathrm{U}(N)}^{\otimes n}\big(\Gamma_u(\mathbf{v};N,2m,\delta')\big)\times  
(\nu_N\otimes\gamma_{\mathrm{U}(N)})\big(\Theta(N,3m,\delta')\big) \\
&= 
\gamma_{\mathrm{U}(N)}^{\otimes (n+1)}\big(\Theta(N,3m,\delta')\cap(\Gamma_u(\mathbf{v};N,3m,\delta')\times\mathrm{U}(N))\big) \\
&\leq  
\gamma_{\mathrm{U}(N)}^{\otimes (n+1)}\big(\big\{(V_1,\dots,V_n,U) \in \mathrm{U}(N)^{n+1}\,|\,\\
&\qquad\qquad\qquad(V_1 U W_1(N)^*,\dots,V_n U W_n(N)^*,UW_{n+1}(N)^*) \\
&\qquad\qquad\qquad\qquad\in \Gamma_\mathrm{orb}(v_1\mathbf{X}v_1^*,\dots,v_n\mathbf{X}v_n^*,\mathbf{X}:(\Xi_i(N))_{i=1}^{n+1}:\mathbf{v}\,;\,N,m,\delta) \big\}\big) \\
&= 
\int_{\mathrm{U}(N)}\gamma_{\mathrm{U}(N)}(dU)\,\gamma_{\mathrm{U}(N)}^{\otimes n}\big(\big\{(V_1,\dots,V_n) \in \mathrm{U}(N)^n\,|\,\\
&\qquad\qquad(V_1 UW_1(N)^*,\dots,V_n UW_n(N)^*,UW_{n+1}(N)^*) \\
&\qquad\qquad\qquad\qquad\in \Gamma_\mathrm{orb}(v_1\mathbf{X}v_1^*,\dots,v_n\mathbf{X}v_n^*,\mathbf{X}:(\Xi_i(N))_{i=1}^{n+1}:\mathbf{v}\,;\,N,m,\delta) \big\}\big) \\
&= 
\int_{\mathrm{U}(N)}\gamma_{\mathrm{U}(N)}(dU)\,\gamma_{\mathrm{U}(N)}^{\otimes n}\big(\big\{(V_1,\dots,V_n) \in \mathrm{U}(N)^n\,|\,\\
&\qquad\qquad(V_1 UW_{n+1}(N)W_1(N)^*,\dots,V_n UW_{n+1}(N)W_n(N)^*,U) \\
&\qquad\qquad\qquad\qquad\in \Gamma_\mathrm{orb}(v_1\mathbf{X}v_1^*,\dots,v_n\mathbf{X}v_n^*,\mathbf{X}:(\Xi_i(N))_{i=1}^{n+1}:\mathbf{v}\,;\,N,m,\delta) \big\}\big) \\
&= 
\int_{\mathrm{U}(N)}\gamma_{\mathrm{U}(N)}(dU)\,\gamma_{\mathrm{U}(N)}^{\otimes n}\big(\big\{(U_1,\dots,U_n) \in \mathrm{U}(N)^n\,|\,\\
&\qquad(U_1,\dots,U_n,U) \in \Gamma_\mathrm{orb}(v_1\mathbf{X}v_1^*,\dots,v_n\mathbf{X}v_n^*,\mathbf{X}:(\Xi_i(N))_{i=1}^{n+1}:\mathbf{v}\,;\,N,m,\delta) \big\}\big) \\
&= 
\gamma_{\mathrm{U}(N)}^{\otimes(n+1)}\big(\Gamma_\mathrm{orb}(v_1\mathbf{X}v_1^*,\dots,v_n\mathbf{X}v_n^*,\mathbf{X}:(\Xi_i(N))_{i=1}^{n+1}:\mathbf{v}\,;\,N,m,\delta)\big),
\end{align*} 
where the fourth line is obtained by \eqref{Eq8} and both the sixth and the seventh due to the right-invariance of the Haar probability measure $\gamma_{\mathrm{U}(N)}$. Hence the desired inequality follows as in the proof of Theorem \ref{T1}. 
\end{proof}

In view of our work \cite{Ueda:Preprint2016} and Voiculescu's liberation theory \cite{Voiculescu:AdvMath99}, a candidate constraint of selecting sequences of multi-matrices may be the way of approximating to $\mathbf{X}_1\sqcup\cdots\sqcup\mathbf{X}_n$ globally, though it probably does not satisfy the exact formula relating $\chi_\mathrm{orb}$ to $\chi$ in general. 

\section*{Acknowledgment} 

The author would like to thank the referee for his or her careful reading and pointing out several typos.

\end{document}